\documentclass[a4paper,12pt]{amsart}

\usepackage{epsfig}
\usepackage{amsthm,amsfonts}
\usepackage{amssymb,graphicx,color}
\usepackage[all]{xy}

\newtheorem{theorem}{Theorem}[section]
\newtheorem{lemma}[theorem]{Lemma}

\newtheorem{definition}[theorem]{Definition}
\newtheorem{example}[theorem]{Example}
\newtheorem{remark}[theorem]{Remark}

\begin{document}

\title[Lipschitz contact equivalence and real analytic functions .
 ]{Lipschitz contact equivalence and real analytic functions .
}

\author[]{Lev Birbrair*}\thanks{*Research supported under CNPq 302655/2014-0 grant and by Capes-Cofecub}
\address{Departamento de Matem\'atica, Universidade Federal do Cear\'a
(UFC), Campus do Picici, Bloco 914, Cep. 60455-760. Fortaleza-Ce,
Brasil} \email{birb@ufc.br}

\author[]{Rodrigo Mendes}
\address{Instituto de ci\^encias exatas e da natureza, Universidade de Integra\c{c}\~ao Internacional da Lusofonia Afro-Brasileira (unilab)
, Campus dos Palmares, Cep. 62785-000. Acarape-Ce,
Brasil} \email{rodrigomendes@unilab.edu.br}

%\thanks{The authors were partially supported by CNPq-Brazil}

\begin{abstract}
We study the properties of the a complete invariant of the analytic function of two variables with respect to the Lipschitz contact equivalence. This invariant is called pizza. We prove that the pizza of real analytic functions has some continuity properties.
\end{abstract}

\maketitle

\section{Introduction}

The paper is devoted to Lipschitz geometry of real analytic functions. In \cite{BFTG} the authors proved that for analytic functions of two variables the notions of Lipschitz $\mathcal{K}$- equivalence and blow-analytic equivalence coincide. The paper \cite{BFGG} is devoted to classification of germs of subanalytic functions of $2$ variables up to Lipschitz $\mathcal{K}$-equivalence. The invariant defined in \cite{BFGG} is called the pizza. Pizza is a partition of the germ of $\mathbb{R}^2$ at zero into a finite sequence of germs of curvilinear triangles. On each triangle (a piece of pizza) one can define a width function. Each width function is an affine function defined on a segment of $\mathbb{Q}$. The sequence of the triangles, equipped with the width functions makes a complete invariant of the germs of subanalytic functions with respect to the Lipschitz $\mathcal{K}$-equivalence. 

Here we consider pizzas of germs of analytic functions. We show that such pizzas of analytic functions have special properties. Namely, we prove that the width function is continuous and that the slope of the corresponding affine function is always positive. Notice, that these properties do not hold for the general case.

We use the language of zones, proposed to us by Vincent Grandjean and Andrei Gabrielov. Zones are some special subsets of the set of arcs, parametrized by the distance to the singular point (so called Valette's link (see \cite{V})).
Namely, a zone is a subset of Valette's link, such that for any two arcs $\gamma_1, \gamma_2$ in the zone, all the arc, belonging to the H\"older triangle, bounded by $\gamma_1$ and $\gamma_2$ also belongs to the zone. We are interested in the zones defined by arcs such that the order of a function $f$ on all these arcs is constant. An arc, belonging to the zone Z is called generic, if it belongs to a H\"older triangle, generated by a pair of the other arcs to the zone, such that the order of contact of these two arcs and the order of contact of any one of these with the given arc is equal to the size of the zone. Our main result is that for an analytic function, any arc belonging to any zone of the constant order is generic. The properties of the pizza of an analytic function follow from this.

\section{Definitions and basic concepts}

\medskip

Consider the set of subanalytic arcs $\gamma:[0,\epsilon)\rightarrow \mathbb{R}^2$, such that $\gamma(0)=0$. Unless othewise specified, we suppose that the arcs are parametrized by distance to origin, i.e., $\|\gamma(t)\|=t$.

\begin{definition}\label{ordonarc}
\emph{Let $f:\gamma \rightarrow \mathbb{R}$ be a subanalytic function. The Newton-Puiseux expansion of the function looks as follows:
\begin{center}
$f(\gamma(t))=a_1t^{\alpha}+o(t^{\alpha}), \ \alpha \in \mathbb{Q}_+, \ a_1 \in \mathbb{R}^*$.
\end{center} 
The value $\alpha$ is called the order of the function on $\gamma$. We use the notation $ord_\gamma f$. }
\end{definition}
In the definition \ref{ordonarc} we suppose that the function $f$ is defined only on the arc $\gamma$. But this definition makes sense if the function is defined on a subanalytic set $X$, such that $\gamma \subset X$.

\begin{definition}(order of tangency)
\emph{ Let $\gamma_1$ and $\gamma_2$ be two subanalytic arcs. Define:
\begin{center}
$tord(\gamma_1,\gamma_2)=ord_\gamma \|\gamma_1(t)-\gamma_2(t)\|$.
\end{center}
 }
\end{definition}

The arcs are called tangents if $tord(\gamma_1,\gamma_2)>1$.

\begin{definition}
\emph{Two arcs $\gamma_1$ and $\gamma_2$ divide the germ of $\mathbb{R}^2$ at the origin into two components. If $\beta=tord(\gamma_1,\gamma_2)>1$ then the closure of the smaller (not containing a half-plane) component is called a \emph{$\beta$-H\"older triangle}. If $tord(\gamma_1,\gamma_2)=1$ then the closure of each of the two components is called a \emph{$1$-H\"older triangle}. The number $\beta \in \mathbb{Q}_+$ is called the \emph{exponent} of the H\"older triangle. We denote by $T(\gamma_1,\gamma_2)$ the H\"older triangle bounded by ``sides"  $\gamma_1$ and $\gamma_2$ }.
\end{definition}

\begin{definition}
\emph{A \emph{zone} in the space of arcs is a subset such that for any H\"older triangle, defined by any two arcs from the zones we have the following property: Let $\gamma$ is another arc such that the germ of $\gamma$ at the origin belongs to the H\"older triangle. Then $\gamma$ belongs to the zone. }
\end{definition}

\begin{example}

\begin{enumerate}
\item \emph{One arc is a zone. This zone are called special.}

\item  \emph{Take any H\"older triangle, then all the arcs, belonging to the interior the triangle form a zone.}

\item \emph{The following set form a zone as well: Take all the arcs $\gamma$ belong to the upper half-plane, such that $tord(\gamma,L_x)>2$, where $L_x$ is the non-negative part of $x$-axis.}
\end{enumerate}

\end{example}

Let $Z$ be a zone, we define the \emph{size of the zone} as the infimum of $tord(\gamma_1,\gamma_2)$, where the arcs $\gamma_1$ and $\gamma_2$ belong to $Z$. We denote the size of $Z$ by $\mu(Z)$.

\begin{remark}
\emph{We say that the size of a zone containing only a single arc is equal to infinity.}
\end{remark}
\begin{definition}
\emph{An arc $\gamma$ is called a finitely presented arc when $\gamma$ can be parametrized by a finite Puiseux expansion, i.e., $\gamma(x)=(x,f(x))$ (or $\gamma(y)=(g(y),y)$), where $f(x)$ (resp. $g(y)$) is a finite Puiseux series}.
\end{definition}

\begin{lemma}\label{finitely presented}
\emph{Let $Z$ be a zone such that $\mu(Z)<\infty$. Then, there exists a finitely presented arc, belonging to this zone. }
\end{lemma}
\begin{proof}
Since $\mu(Z)<\infty$, there exist a pair of arcs $\gamma_1, \gamma_2 \in Z$. By definition of zone, we have that $\gamma \in Z$, for all $\gamma \subset T(\gamma_1,\gamma_2)$. If $tord(\gamma_1,\gamma_2)=1$, then $\gamma_1$ and $\gamma_2$ are not tangent at $0$. It means that there is a line $l$ given by $l(x)=(a,b)x$, such that $l \subset T(\gamma_1,\gamma_2)$. Hence $l \in Z$. Notice that $l$ has a finite presentation. Consider the case where $tord(\gamma_1,\gamma_2)>1$. Then, $\gamma_1$ and $\gamma_2$ are tangents at $0$. We can suppose that $\gamma_1$ and $\gamma_2$ are tangent to $x$-axes. Let  $\gamma_1(x)=(x,y_1(x))$ and $\gamma_2(x)=(x,y_2(x))$ be a Puiseux parametrizations , where $y_i=[0,\delta]\rightarrow \mathbb{R}$ are convergent series with fractional exponents. Notice that if $y_i(x)$ is a finite serie, for some $i$, or $l_1=\{(x,0)\} \in Z$, we are done. Suppose its no so.  Let $y_1(x)=\sum_{i=1}^{\infty}a_ix^{\alpha_i}$ and $y_2(x)=\sum_{i=1}^{\infty}b_ix^{\beta_i}$, where $\alpha_1<\alpha_2<\ldots$ and $\beta_1<\beta_2<\ldots$. Since $\gamma_1 \neq \gamma_2$, there is two possibilities:
\begin{enumerate}
\item There exists the smallest exponent $\alpha_r \neq \beta_r$ such that $\alpha_i=\beta_i$ and $a_i=b_i$ for $i=1,2\ldots,r-1$.
\item There exists the smallest coefficient $a_{r} \neq b_{r}$ such that $\alpha_i=\beta_i$, $i=1,2,\ldots, r$ and $a_i=b_i$, $i=1,2,\ldots,r-1$. 
\end{enumerate}
Consider the possibility (1). Assume $\beta_r<\alpha_r$. 
\begin{itemize}
\item Suppose $b_r>0$. Notice that $y_2(x)-y_1(x)=b_rx^{\beta_r}+o(x^{\beta_r})>0$, for $x$ sufficiently small.  Take $\gamma(x)=(x,y(x))=(x,\sum_{i=1}^{r-1}a_ix^{\alpha_i}+(a_r+\delta)x^{\alpha_r})$, $\delta>0$. Then, $y(x)-y_1(x)=\delta x^{\alpha_r}+o(x^{\alpha_r})>0$ and $y_2(x)-y(x)=b_rx^{\beta_r}+o(x^{\beta_r})>0$, $x$ sufficiently small. Then, $\gamma \subset T(\gamma_1,\gamma_2)$ and $\gamma \in Z$. Hence,$\gamma$ is a finitely presented arc. 
\item Supose $b_r<0$. Notice, in this case, that $y_1(x)-y_2(x)=-b_rx^{\beta_r}+o(x^{\beta_r})>0$. Let a finitely presented arc given by $\gamma(x)=(x,y(x))=(x,\sum_{i=1}^{r-1}a_ix^{\alpha_i}+(b_r+\delta)x^{\beta_r})$, where $0<\delta<-b_r$. Then, $y(x)-y_2(x)=\delta x^{\beta_r}+o(x^{\beta_r})>0$ and $y_1(x)-y(x)=(-b_r-\delta)x^{\beta_r}+o(x^{\beta_r})>0$. Then, $\gamma \in Z$.
\end{itemize}
\item Consider the possibility (2). We have that $y_2(x)-y_1(x)=(b_r-a_r)x^{\alpha_r}+o(x^{\alpha_r})$. We assume that, $b_r-a_r>0$. Otherwise, we take  $y_1(x)-y_2(x)$. Let $\gamma(x)=(x,y(x))=(x,\sum_{i=1}^{r-1}a_ix^{\alpha_i}+(\frac{a_r+b_r}{2})x^{\alpha_r})$. Then, $y(x)-y_1(x)=(\frac{b_r-a_r}{2})x^{\alpha_r}+o(x^{\alpha_r})>0$ and $y_2(x)-y(x)=(\frac{b_r-a_r}{2})x^{\alpha_r}+o(x^{\alpha_r})>0$. Then, $\gamma \in Z$.
Thus, the lemma is proved.
\end{proof}

A zone is called closed if there are two arcs belonging to it, such that the order of contact between them is equal to the size of the zone. Let $Z$ be a closed zone. An arc $\gamma$ is called \emph{generic} for the zone $Z$ if there exist two different arcs $\gamma_1, \gamma_2 \in Z$, belonging to different sides of $\gamma$ such that
\begin{center}
$tord(\gamma,\gamma_1)=tord(\gamma,\gamma_2)=\mu(Z)$.
\end{center}

Let $f:(\mathbb{R}^2,0)\rightarrow (\mathbb{R},0)$ be a germ of a continuous subanalytic function and let $\gamma$ be an arc. We denote by $Z_{\gamma}f$ the zone of the arcs $\tilde{\gamma}$ such that $\gamma \in Z_{\gamma}f$ and 
$ord_{\tilde{\gamma}}f=ord_{\gamma}f$ for any arc $\tilde{\gamma} \in Z_{\gamma}f$.
\begin{remark}
\emph{A arc $\gamma$ is called \emph{generic with respect to $f$} if $\gamma$ is generic for the zone  $Z_{\gamma}f$. }
\end{remark}

A zone $Z$ is called a \emph{monotonicity zone} of $f$ if for any two arcs $\gamma_1, \gamma_2$ such that $ord_{\gamma_1}f=ord_{\gamma_2}f$ and for any arc $\gamma$, belonging to the H\"older triangle, defined by $\gamma_1$ and $\gamma_2$ one has:
\begin{center}
$ord_{\gamma}f=ord_{\gamma_1}f=ord_{\gamma_2}f$.
\end{center}
Let $Z$ be a monotonicity zone of $f$. Let us define $Q_{Z,f}$ as the set of value of $ord_{\gamma} f$ for all the arcs $\gamma \in Z$.

\begin{remark}
\emph{For any monotonicity zone of any continuous subanalytic function $f$, the set $Q_{Z,f}$ is open, closed or semi-open segment in $\mathbb{Q} \cup \infty$.}
\end{remark}

\section{Width function}
Let $Z$ be a monotonicity zone of $f$. We define the \emph{width function} $\mu$, associated to the zone $Z$. Let $q \in Q_{Z,f}$. Let $\gamma \in Z$ be an arc, such that $ord_{\gamma} f=q$. We define $\mu:Q_{Z,f} \rightarrow \mathbb{Q}_+$ by
\begin{center}
$\mu(q)=\mu(Z_{\gamma}f)$.
\end{center}
Since $Z$ is a monotonicity zone, the set of the arcs $\tilde{\gamma}$ in $Z$, such that $ord_{\tilde{\gamma}}f=q$ form a zone for any $q \in Q_{Z,f}$. In other words, $\mu$ is well defined.

\begin{remark}
\emph{The set of monotonicity zones is semi ordered by inclusions. That is why one can consider a maximal monotonicity zone, containing an arc $\gamma$. Notice, that on a maximal monotonicity zone the function $\mu$ is well defined because the maximal monotonicity zone is a monotonicity zone.}
\end{remark}
Notice, that for a continuous subanalytic function, all the maximal monotonicity zones are closed (it follows from \cite{BFGG}).
\begin{remark}
\emph{Let $f{:}(\mathbb{R}^2,0)\rightarrow (\mathbb{R},0)$ be a germ of an analytic function. Writing $f=f_m+f_{m+1}+\ldots+f_k+\ldots$, where $f_k$ is a homogeneous polynomial of degree $l$, we have that $\mu(m)=1$.  Notice that all the arcs $\gamma$ such that $ord_\gamma f=m$ are generic with respect to $f$. Moreover, if $f$ is reduced, the equality $\mu(m)=1$ provides a {subanalytic} bi-Lipschitz invariance of the multiplicity $m$ (a particular case of the theorem of Trotman and Risler (see \cite{TR})).   }
\end{remark}

\section{Main result} 
Let $f:(\mathbb{R}^2,0) \rightarrow (\mathbb{R},0)$ be a germ of a subanalytic function.
\begin{definition}\label{maximal}
\emph{An arc $\gamma$ is called of \emph{maximal order} with respect to $f$, if:}
\begin{enumerate}
\item  \emph{For any arc $\tilde{\gamma}$ such that $tord(\tilde{\gamma},\gamma)\geq {\mu(Z_{\gamma}f)}$ we have }
\begin{center}
$ord_{\tilde{\gamma}}(f)=ord_{\gamma}(f)$;
\end{center}
\item \emph{If $tord(\tilde{\gamma},\gamma) < {\mu(Z_{\gamma}f})$ we have $ord_{\tilde{\gamma}}(f)<ord_{\gamma}(f)$, for all $\tilde{\gamma} \subset Z$, where $Z$ is a maximal monotonicity zone of $f$ containing $\gamma$}.
\end{enumerate}
\end{definition}
In this case, we say that $Z_{\gamma}f$ is \emph{a zone of maximal order}. Notice that a zone of maximal order is completely determined by $\mu(Z_{\gamma}f)$.

\begin{remark}
{\emph{Let $f:(\mathbb{R}^2,0) \rightarrow (\mathbb{R},0)$ be a germ of a analytic function. For all arcs $\gamma \subset f^{-1}(0)$, we have that $\mu(Z_{\gamma}f)=\infty$. In other words, $f^{-1}_{\mathbb{R}}(0)$ is a finite union of special zones. All the special zones are formed by arcs of maximal order with respect to f. }}
\end{remark}
{
On the other hand, an arc of the maximal order is not necessarily special. For instance, consider the function $f(x,y)=x^8+y^6$. The $x$-axes is a arc of maximal order, where $\{(x,0)\} \nsubseteq f^{-1}(0)$.}

\begin{lemma}
\emph{There are finitely many zones, containing the arcs of maximal order.}
\end{lemma}
\begin{proof}
Notice that the analytic functions one has the Fukuda conical structure theorem (see \cite{F}). It means that for sufficiently small $\epsilon$, $|f||_{\mathbb{S}^1(0,\epsilon)}$ has finitely many local maximum and minimums. On the other hand, any zone containing the arcs of the maximal order with respect to $f$ corresponds to a local minimum of $|f||_{\mathbb{S}^1(0,\epsilon)}$. That is why the number of these zones is finite.
\end{proof}
\begin{theorem}\label{genericity}
\emph{Let $f$ be a germ of a analytic function. Then, any arc $\gamma$ is generic with respect to $f$.}

\end{theorem}
\begin{proof}
By the previous remarks, there is two cases where the genericity of $\gamma$ with respect to $f$ were obtained:
\begin{enumerate}
\item $Z_{\gamma} f=\{\gamma\}$;
\item $\mu(Z_{\gamma} f)=1$.
\end{enumerate}
Then, let $1<\mu(Z_{\gamma} f) <\infty$. Using a orthogonal transformation on $\mathbb{R}^2$, we get that $Z_{\gamma} f$ is tangent to $x$-axes. We consider the case where $\gamma$ is parametrized by $\gamma(x)=(x,b_1x^{\beta_1}), \ b_1 \neq 0, \ \beta_1>1$. Let $f^{-1}(0)=\cup_{i=1}^rB_i$, where $B^0_i(x)=(x,a_ix^{\frac{p_i}{q_i}}), \ i=1,\ldots,r$ are initial parametrizations of the branches $B_i$. We consider the following cases:

\medskip

A) For all $i$, $a_i$ is not real ;

\medskip

Proof of genericity: The composition $f(\gamma(x))$ can be described as follows:
\begin{center}
$f(\gamma(x))=ax^{v_1}+o(x^{v_1})$,
\end{center}
where $a=f_0(1,b_1)$, $f_0$ is a quasi-homogeneous polynomial with weights $(1,\beta_1)$ and $v_1$ is the weighted degree associated. The notation $o(x^{v_1})$ mean terms of order $>v_1$. Suppose that $\beta_1=\frac{p_l}{q_l}$, for some $l$. Let $\gamma^+_i$ and $\gamma^-_i$ be arcs defined as follows: 
\begin{center}
$\gamma^+_i(x)=(x,(b_1+\delta)x^{\beta_1})$ and $\gamma^-_i(x)=(x,(b_1-\delta)x^{\beta_1}), \ \delta>0$.
\end{center}
Since $a_i$ is not real, the polynomial $P_l(x)=f_0(1,b_1+x)$ do not have real roots. Hence $ord_{\gamma^+_i} f=ord_{\gamma^-_i} f=v_1$, Moreover, 
\begin{center}
$tord(\gamma^+_i,\gamma^-_i)=ord_x\|(x,(b_1-\delta)x^{\frac{p_l}{q_l}})-(x,(b_1+\delta)x^{\frac{p_l}{q_l}})\|=ord_x\|-2\delta x^{\frac{p_l}{q_l}}\|=\frac{p_l}{q_l}=\beta_1$
\end{center}
 
and $tord(\gamma,\gamma^-_i)=tord(\gamma^+_i,\gamma)=tord(\gamma^+_i,\gamma^-_i)$. Notice that if $\beta_0 < \beta_1$ then $i+\beta_0j<i+\beta_1 j=v_1$. Hence, $\mu(Z_{\gamma}f)=\beta_1$ and $\gamma$ is generic. Otherwise, if $\beta_1 \neq \frac{p_i}{q_i}$, for all $i$, we have that $P_l$ is a monomial and the argument is the same. Then $\gamma$ is generic with respect to $f$.

\medskip

B) For some $l$, $a_l$ is real. 

\medskip

Proof of genericity:

Let $\beta_1=\frac{p_l}{q_l}$. We present the serie $f \circ \gamma$ in the following form: $f(\gamma(x))=x^{v_1}P_l(b_1)+h(\gamma(x))$, where $ord_xh(x)>v_1$. If $P_l(b_1)\neq 0$, by continuity of $P_l$, there is $\delta>0$, sufficiently small such that $P_l(s) \neq 0$, $\forall s \in [b_1-\delta,b_1+\delta]$. Let $\gamma^-_i$ and $\gamma^+_i$ be the arcs defined as above. Then, we get that $ord_{\gamma^+_i} f=ord_{\gamma^-_i} f=v_l$ and $tord(\gamma_i^-,\gamma_i^+)=\beta_1$. Hence, $\gamma$ is generic. 

Let $P_l(b_1)=0, \ b_1 \neq 0$ (the case where $b_1=0$ can be considered in the same way). We have that $\gamma(x)=(x,b_1x^{\frac{p_l}{q_l}})$ is the real initial parametrization of the some branch of $f^{-1}(0)$. Then:
\begin{equation}\label{equation 1}
f(x,b_1x^{\frac{p_l}{q_l}})=x^{v_1}P_l(b_1)+h(\gamma(x))=0
+h(\gamma(x))=h(\gamma(x)),
\end{equation}
where $P_l(b_1)=f_0(1,b_1)$. Notice that we may write:

\begin{center}
$f(x,y)=f_0(x,y)+h(x,y)=\displaystyle \sum_{i+\frac{p_l}{q_l}j=v_l}a_{ij}x^iy^j+\displaystyle \sum_{i+\frac{p_l}{q_l}j>v_l}a_{ij}x^iy^j$.
\end{center}

Hence, the expression (\ref{equation 1}) is given by

\medskip

$f(x,b_1x^{\frac{p_l}{q_l}})=\displaystyle \sum_{i+\frac{p_l}{q_l}j>v_l}a_{ij}x^i(b_1x^{\frac{p_l}{q_l}})^j=\displaystyle \sum_{i+\frac{p_l}{q_l}j>v_l}a_{ij}b_1^jx^{i+{\frac{p_l}{q_l}j}}=x^{v_l}(\sum_{i+\frac{p_l}{q_l}j>v_l}a_{ij}b_1^jx^{i+{\frac{p_l}{q_l}j-v_l}})=x^{v_l}g(x)$.

\medskip

Since $\mu(Z_{\gamma}f)<\infty, \ \ ord_x(g(x)) < \infty$. Then $ord_{\gamma}f=v_l+ord_xg(x) $. In order to obtain the genericity of $\gamma$, we consider the following deformation:
\begin{center}
$\gamma^{s,T}(x)=(x,b_1x^{\frac{p_l}{q_l}}+sx^T), \ s \in \mathbb{R}^*, \ T \in (\frac{p_l}{q_l},\infty) \cap \mathbb{Q}$. 
\end{center}

\medskip

Claim 1:

$f_0(\gamma^{s,T}(x))\neq 0$ and $ordf_0(\gamma^{s,T}(x))$ depends of $T$ as affine function, $s \neq 0$.

\medskip

Proof of claim 1: We calculate the function $f_0(\gamma^{s,T}(x))$ as follows:

\medskip

$f_0(\gamma^{s,T}(x))=\displaystyle \sum_{i+\frac{p_l}{q_l}j=v_l}{a}_{ij}x^i(b_1x^{\frac{p_l}{q_l}}+sx^T)^j=\displaystyle \sum_{i+\frac{p_l}{q_l}j=v_l}{a}_{ij}x^i(b_1x^{\frac{p_l}{q_l}}(1+\frac{s}{b_1}x^{T-\frac{p_l}{q_l}}))^j$

\medskip

$=\displaystyle \sum_{i+\frac{p_l}{q_l}j=v_l}{a}_{ij}x^i[b_1^jx^{\frac{p_l}{q_l}j}](1+\frac{s}{b_1}x^{T-\frac{p_l}{q_l}})^j=\displaystyle \sum_{i+\frac{p_l}{q_l}j=v_l}{a}_{ij}b_1^jx^{i+{\frac{p_l}{q_l}j}}(1+\frac{s}{b_1}x^{T-\frac{p_l}{q_l}})^j=$

\medskip

$x^{v_l}(\displaystyle \sum_{i+\frac{p_l}{q_l}j=v_l}{a}_{ij}b_1^j+ \displaystyle \sum_{i+\frac{p_l}{q_l}j=v_l}j{a}_{ij}b_1^j\frac{s}{b_1}x^{T-\frac{p_l}{q_l}}+\ldots+\displaystyle \sum_{i+\frac{p_l}{q_l}j=v_l}{a}_{ij}b_1^j(\frac{s}{b_1})^jx^{j(T-\frac{p_l}{q_l})} )=$

\medskip

$x^{v_l}(P_l(b_1)+ \displaystyle \sum_{i+\frac{p_l}{q_l}j=v_l}j{a}_{ij}b_1^{j-1}sx^{T-\frac{p_l}{q_l}}+\ldots+\displaystyle \sum_{i+\frac{p_l}{q_l}j=v_l}{a}_{ij}s^jx^{j(T-\frac{p_l}{q_l})} )=$

\medskip

$x^{v_l}( 0+\displaystyle \sum_{i+\frac{p_l}{q_l}j=v_l}j{a}_{ij}b_1^{j-1}sx^{T-\frac{p_l}{q_l}}+\displaystyle \sum_{i+\frac{p_l}{q_l}j=v_l}\frac{j(j-1)}{2}{a}_{ij} b_1^{j-2}s^2x^{2({T-\frac{p_l}{q_l}})}\ldots+\displaystyle \sum_{i+\frac{p_l}{q_l}j=v_l}{a}_{ij}(sx^{(T-\frac{p_l}{q_l})})^j )$

\medskip

$=x^{v_l}(\displaystyle P_1(b_1)sx^{T-\frac{p_l}{q_l}}+P_2(b_1)(sx^{T-\frac{p_l}{q_l}})^2\ldots+P_{j_r}(b_1)(sx^{(T-\frac{p_l}{q_l})})^{j_r} )$,

where $j_1<j_2< \ldots <j_r$ are all the ordinates of the points $(i,j)$ such that $i+\frac{p_l}{q_l}j=v_l$. Thus, we have 
\begin{equation}\label{genericc}
ord_xf_0(\gamma^{s,T}(x))=v_l+j_0(T-\frac{p_l}{q_l}),
\end{equation}
for some value of $j_0$, $T \in (\frac{p_l}{q_l},+\infty) \cap \mathbb{Q}$.

\medskip

\noindent The claim 1 is proved.

\medskip

Claim 2: The following identity is true:
\begin{center}
$ord_x h(\gamma^{s,T}(x))=ord_{\gamma}f$,
\end{center}
for all $T \in (\frac{p_l}{q_l}, +\infty)$.
 
\medskip

Proof of claim 2. 

\medskip

$h(\gamma^{s,T}(x))=\displaystyle \sum_{i+\frac{p_l}{q_l}j>v_l}\tilde{a}_{ij}x^i(b_1x^{\frac{p_l}{q_l}}+sx^T)^j=\sum_{i+\frac{p_l}{q_l}j>v_l}\tilde{a}_{ij}x^i(b_1x^{\frac{p_l}{q_l}}+sx^T)^j=$

\medskip

$\sum_{i+\frac{p_l}{q_l}j>v_l}\tilde{a}_{ij}x^i(b_1x^{\frac{p_l}{q_l}}+b_1x^{\frac{p_l}{q_l}}\frac{s}{b_1}x^{T-\frac{p_l}{q_l}})^j=\sum_{i+\frac{p_l}{q_l}j>v_l}\tilde{a}_{ij}x^i(b_1x^{\frac{p_l}{q_l}}(1+\frac{s}{b_1}x^{T-\frac{p_l}{q_l}}))^j=$

\medskip

$x^{v_l}(\sum_{i+\frac{p_l}{q_l}j>v_l}\tilde{a}_{ij}b_1^jx^{i+\frac{p_l}{q_l}j-v_l}(1+\frac{s}{b_1}x^{T-\frac{p_l}{q_l}})^j)=$

$x^{v_l}(\sum_{i+\frac{p_l}{q_l}j>v_l}\tilde{a}_{ij}b_1^jx^{i+\frac{p_l}{q_l}j-v_l}+\sum_{i+\frac{p_l}{q_l}j>v_l}\tilde{a}_{ij}j\frac{s}{b_1}b_1^jx^{i+\frac{p_l}{q_l}j-v_l+(T-\frac{p_l}{q_l})}+\ldots$ \newline

\medskip

$
+\sum_{i+\frac{p_l}{q_l}j>v_l}\tilde{a}_{ij}b_1^j\frac{{s_1}^j}{{b_1}^j}x^{i+\frac{p_l}{q_l}j-v_l+j(T-\frac{p_l}{q_l})} \ )=$

\medskip
\begin{equation}\label{exp} 
x^{v_l}(g(x)+g_1^T(x)+\ldots+g_j^T(x)),
\end{equation}
where $ord_xg(x)<ord_xg_1^T(x)<\ldots<g_j^T(x)$. Then

\medskip
\begin{center}
$ord_x h(\gamma^{s,T}(x))=v_l+ord_xg(x)=ord_{\gamma}f$.
 
\end{center}
Thus, the claim 2 is proved.
 
 \medskip
 
 Now, we consider, for generic parameter $s \in \mathbb{R}^*$, the equation
 \begin{equation}\label{exponent}
 ord_xf_0(\gamma^{s,T}(x))=ord_{\gamma} f.
 \end{equation}
 \begin{center}
 $ \Rightarrow v_l+j_0(T-\frac{p_l}{q_l})=v_l+ord_xg(x)$ 
 $ \Rightarrow T-\frac{p_l}{q_l}=\frac{ord_xg(x)}{j_0}$.  $ \therefore T_0=\frac{ord_xg(x)}{j_0}+\frac{p_l}{q_l}$.
 \end{center}
 It means that, for this $T_0$, $ord_{\gamma^{s,T_0}}f$ is equal to $ord_{\gamma} f$ and then, the size of the zone $Z_{\gamma} f$ is equal to $T_0$. In order to obtain the genericity of $\gamma$, it is enough to take $\gamma^{s,T_0} , \gamma^{-s,T_0}$, $s \in \mathbb{R}^*$ such that $ord_xf_0(\gamma^{\pm s,T}(x))=ord_{\gamma} f$.

\medskip

Consider the general case. By lemma \ref{finitely presented}, we may consider an arc $\gamma \in Z_{\gamma} f$ parametrized by $\gamma(x)=(x,b_1x^{\beta_1}+b_2x^{\beta_2}+\ldots+b_rx^{\beta_r})$, where $\beta_1<\beta_2<\ldots <\beta_r$. The parametrization satisfies the following:  
\begin{equation}\label{algorithm}
 ord_{\gamma_{\beta_{1}}}f <  ord_{\gamma_{\beta_{2}}}f < \ldots < ord_{\gamma_{\beta_{r}}}f= ord_{\gamma}f, 
 \end{equation} 
where  $\gamma_{\beta_i}(x)=(x,b_1x^{\beta_1}+b_2x^{\beta_2}+\ldots+b_ix^{\beta_i})$. We say that $\gamma(x)$ is a minimal parametrization providing the order of the zone $Z_{\gamma} f$. In this way, we get that $\gamma(x)$ is a initial parametrization for some branch $B_i {\subset} f^{-1}_{\mathbb{C}}(0)$, where the last exponent $\beta_r$ is obtained from the step $r$ of the Newton algorithm. Recursively, we have:
\begin{equation}
f(\gamma(x))=f_r(\gamma(x))+h_r(\gamma(x))=0+h_r(\gamma(x)).
\end{equation}
The previous equation providing $\mu(Z_\gamma f)$ is obtained now considering $\gamma^{{\pm}s,T}(x)=(x,b_1x^{\beta_1}+b_2x^{\beta_2}+\ldots+b_rx^{\beta_r}\pm sx^T), \ T {\in} (\beta_r, +\infty) \cap \mathbb{Q}$. In the same way, we obtain that $ord_xf_r(\gamma^{s,T}(x)) \neq 0$ depends of $T$ as affine function for generic $s$. Thus, resolving the equation  $ord_xf_r(\gamma^{s,T}(x))=ord_{\gamma} f$, we show that the genericity property is obtained by the previous arguments.

\end{proof}

\begin{remark}
\emph{Notice that if $\gamma$ is an arc of maximal order, $\gamma$ parametrize $B_i \cap \mathbb{R}^2$. Otherwise, when the zone $Z_\gamma f$ do not have a maximal order, there exists $s_0$ such that $ord_xf_r(\gamma^{s_0,T_0}(x))>ord_{\gamma} f$ where $\mu(Z_{\gamma} f){=}T_0$. It means that $T_0$ is a next Puiseux exponent of the parametrization $B_i(x)$.}
\end{remark}
\section{Continuity of Pizza}
We are going to remind the main result of the paper \cite{BFGG}. Let $f:\mathbb{R}^2,0 \rightarrow \mathbb{R},0$ be a germ of a continuous subanalytic function. Then there exists a partition of the germ of $\mathbb{R}^2$ into H\"older triangles $T_i$ elementary (see \cite{BFGG}, page 3) with respect to $f$ such that for each triangle the width function is well defined. Moreover on each $T_i$ the width function $\mu_i: Q(T_i) \rightarrow \mathbb{Q}$ is  an affine function. 
\begin{definition}
\emph{The width function, associated to a pizza of a germ of a continuous subanalytic function is called continues if for the adjacent triangles $T_1$ bounded by $\gamma_1$ and $\gamma$ and $T_2$, bounded by $\gamma_2$ and $\gamma$ (where $\gamma$ is a commom boundary) we have the following: Let $ord_{\gamma} f=q$. Then $\mu_1(q)=\mu_2(q)$. }
\end{definition}
Notice, that for a general continues subanalytic function, the width function $\mu$ is not necessarily continues. For instance, consider the function $f(x,y)=min\{x^2, y^3\}$. Let $\gamma(t)=(t^{\frac{3}{2}},t)$. Then, $ord_\gamma f=3$. However, $\mu$ is not continues because $\mu_1(3)=1$ and $\mu_2(3)=\frac{3}{2}$.
\begin{theorem}
\emph{The width function associated to a pizza of the germ of an analytic function is continuous}.

\end{theorem}
\begin{proof}
Suppose that  we constructed a pizza, associated to the function $f$. Consider a H\"older triangle of the pizza. If $\gamma$ is an arc, such that $ord_{\gamma} f$ is not equal to the extremal values of the corresponding segment, then the function $\mu$ is continuous function because it is affine. Suppose $q$ is an extremal value of the segment. Then the arc $\gamma$ must belong to the zone, containing of the arcs $\tilde{\gamma}$, where $ord_{\gamma} f=q$. But in the zone all the arcs are generic (see theorem \ref{genericity}). That is why there exists an arc $\tilde{\gamma}$, situated on the left side of $\gamma$ such that $ord_{\tilde{\gamma}} f=q$ and there exists a arc $\tilde{\tilde{\gamma}}$, situated on the right side of $\gamma$, such that $ord_{\tilde{\tilde{\gamma}}}f=q$. Then,
\begin{center}
$tord(\tilde{\gamma},\tilde{\tilde{\gamma}})=tord(\tilde{\gamma},\gamma)=tord(\gamma,\tilde \tilde{\gamma})=\mu(Z_{\gamma} f)$.
\end{center}
 It means that the values of $\mu(q)$ are equal for the two triangles, where $\gamma$ is a boundary arc.
\end{proof}

\section{Positive slope}

\begin{theorem}
\emph{Let $f:(\mathbb{R}^2,0)\rightarrow (\mathbb{R},0)$ be a germ of a analytic function. Consider $Z$ be a maximal monotonicity zone with respect to $f$ such that $Q_{f,Z}$ is not a point. Then, the width function associated is a monotone increasing function}.
\end{theorem}

\begin{proof}
Let $Q_{f,Z}$ be a segment corresponding to the maximal monotonicity zone of $f$. Suppose that $\mu(Z)=1$. Let
\begin{center}
$f=f_m+f_{m+1}+\ldots$,
\end{center}
where $f_j$ is a homogeneous polynomial of degree $j$. We write $f_m^{-1}(0)=\bigcup_{i=1}^m l_i$, where $l_i$ is a line in $\mathbb{C}^2$. We know that $\mu(ord_l f)=1$, when $l \neq l_i$, for all $i$. Since $Q_{f,Z}$ is not a point, there is $l_i \subset Z$, for some $i$. Then, $m=ord_l f<ord_{l_i} f$ and $1=\mu(ord_{l}f)<\mu(ord_{\tilde{l}}f)$. Since $\mu$ is affine in $Q_{f,Z}$, it is follow that $\mu|_{Q_{f,Z}}$ is monotone increasing. Let $\mu(Z)>1$. As before, we suppose, without loss of generality, that $Z$ is tangent to $x$-axes. Since $Z$ is a maximal monotonicity zone, $Q_{f,Z}$ is a closed segment. Let $q$ be the upper extremal of this segment. Consider $\gamma$ be an arc with minimal parametrization 
\begin{center}
$\gamma(x)=(x,c_1x^{\alpha_1}+c_2x^{\alpha_2}+\ldots+c_rx^{\alpha_r})$,
\end{center}
where $ord_{\gamma} f=q$. Let $T_0=\mu(Z_{\gamma} f) \in (\alpha_r,+\infty) \cap \mathbb{Q}$. Notice that for all $T \in (\alpha_r,T_0)$, we have $ord_{\gamma^{s,T}}f<ord_{\gamma^{s,T_0}}f=ord_\gamma f$, where $\gamma^{s,T}(x)=\gamma(x)+(0,sx^T)$. Moreover, $ord_{\gamma^{s,T}}f$ depends of $T$ as affine function. Then $\mu(Z_{\gamma^{s,T}} f)=T$. Hence,  

\begin{center}
 $ord_{\gamma^{s,T}}f<ord_{\gamma^{s,T_0}}f \Rightarrow \mu(Z_{\gamma^{s,T}} f)=T<T_0=\mu(Z_{\gamma^{s,T_0}} f)$.
\end{center}
Thus, since $\mu$ is affine in $Q_{f,Z}$, it follow that $\mu$ is a monotone increasing function.

\end{proof}

\noindent{\bf Acknowledgements}.  We would like to thank Andrei Gabrielov and Vincent Grandjean for their interest on this work, interesting questions and remarks.

\end{document}